\documentclass[12pt]{amsart}
\usepackage{amscd}
\usepackage{amsmath}
\usepackage{verbatim}
\usepackage[normalem]{ulem}
\usepackage{longtable}

\usepackage[all, cmtip,arrow]{xy}

\usepackage{pb-diagram,pb-xy}

\usepackage{amssymb}

\numberwithin{equation}{section}

\newtheorem{thm}[equation]{Theorem}
\newtheorem{prop}[equation]{Proposition}

\newtheorem{lem}[equation]{Lemma}

\theoremstyle{definition}
\newtheorem{rem}[equation]{Remark}
\newtheorem{example}[equation]{Example}

\newtheorem{ntt}[equation]{}

\newcommand{\Mot}{\mathcal{M}}

\newcommand{\SB}{\mathop{\mathrm{SB}}}
\newcommand{\w}{\omega}

\renewcommand{\Im}{\mathop{\mathrm{Im}}}

\newcommand{\CH}{\mathop{\mathrm{CH}}\nolimits}

\newcommand{\PSL}{\operatorname{\mathrm{PSL}}}

\newcommand{\PGSp}{\operatorname{\mathrm{PGSp}}}

\newcommand{\Spin}{\operatorname{\mathrm{Spin}}}

\newcommand{\PGO}{\operatorname{\mathrm{PGO}}}

\newcommand{\GL}{\operatorname{\mathrm{GL}}}

\newcommand{\SL}{\operatorname{\mathrm{SL}}}

\newcommand{\E}{\mathrm{E}}

\newcommand{\Ch}{\mathop{\mathrm{Ch}}\nolimits}

\newcommand{\res}{\mathop{\mathrm{res}}\nolimits}

\newcommand{\cdim}{\mathrm{cdim}}

\newcommand{\Char}{\mathop{\mathrm{char}}\nolimits}

\newcommand{\zz}{\mathbb{Z}}
\newcommand{\F}{\mathrm{F}}
\newcommand{\G}{\mathrm{G}}
\newcommand{\A}{\mathrm{A}}
\newcommand{\B}{\mathrm{B}}

\newcommand{\qq}{\mathbb{Q}}

\newcommand{\C}{\mathrm{C}}

\newcommand{\Spec}{\operatorname{Spec}}
\newcommand{\End}{\operatorname{End}}
\newcommand{\Hom}{\operatorname{Hom}}

\newcommand{\Aut}{\operatorname{Aut}}

\newcommand{\op}{\mathrm{op}}

\newcommand{\pt}{\mathrm{pt}}

\newcommand{\Ker}{\operatorname{Ker}}

\newcommand{\D}{\mathrm{D}}

\newcommand{\trunc}{\mathrm{trunc}}

\renewcommand{\phi}{\varphi}

\usepackage[hypertex]{hyperref}

\title
[Rost motives, affine varieties, and classifying spaces]
{Rost motives, affine varieties, and classifying spaces}


\author
[Victor Petrov]
{Victor Petrov}

\author
[Nikita Semenov]
{Nikita Semenov}

\address{Semenov:
Institut f\"ur Mathematik, Johannes Gutenberg-Universit\"at
Mainz, Staudingerweg 9, D-55128, Mainz, Germany}
\email{semenov@uni-mainz.de}

\address{Petrov: St.~Petersburg Department of Steklov Mathematical Institute,
Russian Academy of Sciences,
Fontanka 27,
191023 St.~Petersburg,
Russia}
\email{victorapetrov@googlemail.com}

\thanks{Results of Sections 6--8 were obtained under the support of RSCF grant 14-11-00456. The second author
acknowledges the support of the SPP 1786 ``Homotopy theory and algebraic geometry'' (DFG)}

\begin{document}

\begin{abstract}
In the present article we investigate ordinary and equivariant Rost motives. We provide an equivariant motivic
decomposition of the variety $X$ of full flags of a split semisimple algebraic group over
a smooth base scheme,
study torsion subgroup of the Chow group of twisted forms of $X$, define some equivariant Rost motives over a field and ordinary Rost motives
over a general base scheme, and relate equivariant Rost motives with classifying spaces of some algebraic groups.
\end{abstract}

\maketitle

\section{Introduction}

The (generalized) Rost motives play an essential role in the proofs of the Milnor and of the Bloch--Kato conjectures
(see \cite{Ro07}, \cite{Vo03}, \cite{Vo11} and references there). Rost motives are uniquely determined by pure symbols in the Galois cohomology
$H_{et}^n(F,\mu_p^{\otimes n})$ of a field $F$, where $n$ is a natural number
and $p$ is a prime number. In particular, a Rost motive depends on two parameters $n$ and $p$. 

In Section~\ref{secaff} we will show that the Rost motives corresponding to {\it small values} of parameters $(n,p)$ can be
identified in the category of motives of Voevodsky with the motives of certain concrete affine varieties.
These identifications resemble the exceptional isomorphisms between small finite groups.

Namely, there exists a general principle in Mathematics related to the phenomenon of small dimensions.
For example, it is well-known that in dimensions bigger than $4$ there are no regular polytopes apart
from simplices, hyper-cubes and hyper-octahedra (which are dual to hyper-cubes). On the other hand, there exist famous beautiful $3$- and $4$-dimensional
regular polytopes, namely, dodecahedron, icosahedron, $24$-cell, $120$-cell, and $600$-cell.
                                        
The same principle is known in other areas of Mathematics. Here is a short outline.

$\bullet$ When studying finite simple groups, one finds a finite number of exceptional groups, which are called
{\it sporadic groups}. There are no sporadic groups of order bigger than $2^{180}$. Thus, a ``small dimension'' in
this example is big, but concrete and finite.

$\bullet$ In the theory of simple {\it Lie groups} and {\it Lie algebras} there are exceptional series of groups and algebras
with the last one of rank $8$ (type $\E_8$). If one would classify simple complex Lie algebras
of rank bigger than $8$, then one would miss all exceptional cases and find the algebras of classical types only.

$\bullet$ In the theory of {\it Jordan algebras} all simple Jordan algebras of dimension bigger than $27$ are special,
and there exists an exceptional Jordan algebra of dimension $27$ (Albert algebra).

$\bullet$ Finally, when studying finite groups, one finds {\it exceptional isomorphisms} between finite groups of small order,
like famous isomorphisms between small $\PSL$ groups over finite fields and alternating groups,
or $\PSL(3,4)$ and Mathieu group $M_{21}$. Moreover, as a consequence of these isomorphisms one finds exceptional
behaviour of the {\it Schur multipliers} of finite groups of small order.

\medskip

One of the goals of the present article is to give another illustration of this principle in the context
of the algebraic geometry.

The amazing identification of the Rost motive with affine varieties mentioned above allows us to define the {\it equivariant} Rost motives for small parameters $(n,p)$.
We start first developing general equivariant methods in Section~\ref{seclift}, where we provide complete $G$-equivariant motivic decompositions of
the variety of full flags of $G$, where $G$ is a split semisimple algebraic group over a field (Theorem~\ref{tlift}).
These motivic decompositions turn out to be related to Kac's degrees of generators of a certain ideal $I_p$ from \cite{Kac85}.
Thus, we give a geometric interpretation of these degrees, as opposed to an algebraic description using combinatorics
of the Weyl group given in \cite{Kac85}.

Applying our general formula of Theorem~\ref{tlift} to varieties of types $\G_2$ and $\F_4$ (where the ordinary Rost
motives appear) and using the identification of small Rost motives with affine varieties we find a relation
between the {\it equivariant Rost motives} and the {\it classifying spaces} of some
algebraic groups. For example, the $\G_2$-equivariant Rost motive for $(n,p)=(3,2)$ is $B\SL_3$ and
the $\F_4$-equivariant Rost motive for $(n,p)=(3,3)$ is $B\Spin_9$. We explain this in Section~\ref{sec6}.
We also remark that recently Pirutka and Yagita used the classifying spaces of some simple simply connected
algebraic groups to provide a counter-example to the integral Tate conjecture over finite fields (see \cite{PY14}).

Moreover, a more detailed analysis shows that there is a close relation between the {\it torsion subgroup}
of the ordinary Rost motive and Kac's degrees of the ideal $I_p$. We establish this relation in general
(not only for types $\G_2$ and $\F_4$) in Section~\ref{sec7}. In particular, it follows from our results that
the Chow group $\CH(E/B)/p$ is {\it finitely generated}, where $E$ is a generic $G$-torsor over a field, $B$ is a Borel subgroup of $G$
and $p$ is a prime number.
Notice that this fact is quite unexpected, since e.g. there exist examples of small-dimensional quadrics due
to Karpenko and Merkurjev with {\it infinitely generated} torsion in the Chow group.
We remark also that the original title of Voevodsky's article \cite{Vo03} (where he proves the Milnor
conjecture) was ``On $2$-torsion in motivic cohomology'', since precisely the structure of the torsion part
was crucial in his proof.

Besides, we provide in Section~\ref{sec7} a motivic decomposition of $E/B$, where $B$ is a Borel subgroup
of a split semisimple algebraic group $G$ over a field $F$ and $E$ is a $G$-torsor over an arbitrary
smooth base scheme. We remark that we are not aware of any other similar result, where the base scheme is not field.

Finally, the identification of small Rost motives with affine varieties allows to define ordinary small {\it Rost motives
over an arbitrary smooth base scheme}. We do it in Section~\ref{rostbase}, where we also show that the Rost motives
over an arbitrary scheme have in general more torsion than over fields.

Besides this, we remark that our equivariant motivic decompositions from Section~\ref{seclift} automatically
provide decompositions of nil-Hecke algebras modulo a prime $p$ and give new information about its modular
representations (see \cite{NPSZ15}).

{\bf Acknowledgments.}
We would like to thank Nikolai Vavilov for his very interesting lectures on exceptional objects in
algebra and geometry given in the Chebyshev Laboratory of the St. Petersburg State University in 2013, which
influenced the present paper. The second author also thanks Thomas Friedrich for his talk about
reductive spaces in small dimensions given at the Mainz University in January 2014 where he explained
some exceptional isomorphisms in the context of differential geometry.

\section{Definitions and notation}
\begin{ntt}[Algebraic groups]
Let $F$ be an arbitrary field. In the present article we consider linear algebraic groups over $F$.
The basic notion of the theory of linear algebraic groups
(e.g. Weyl group, Borel subgroups, fundamental weights, etc.) can be found in many books, e.g. in \cite{Springer}.

In the article we freely use the notion of twisted flag varieties, torsors, Galois cohomology, Milnor $K$-theory, as
well as some classical constructions of algebraic groups related to Pfister forms, octonions and Albert algebras,
see \cite{Inv}, \cite{GMS03}.

Moreover, we associate with a semisimple group $G$ of inner type a set of prime numbers, called {\it torsion primes}.
Namely, we define this set as the union of all torsion primes of all simple components of $G$ and for a simple $G$
of inner type this set is given in the following table:
\begin{center}
\begin{longtable}{c|c}
Dynkin type & Torsion primes\\
\hline
$\A_m$ & $p\mid m+1$\\
$\B_m,\C_m,\D_m,\G_2$ & $2$\\
$\F_4,\E_6,\E_7$ & $2,3$\\
$\E_8$ & $2,3,5$
\end{longtable}
\end{center}
\end{ntt}

\begin{ntt}[Motives]
For a smooth variety $X$ over $F$ we consider its Chow ring $\CH(X)$ of algebraic cycles modulo rational equivalence,
see \cite{Ful}. Sometimes we have a fixed prime $p$ and use the notation $\Ch(X):=\CH(X)\otimes\zz/p$.
We do not write $p$ in the notation of $\Ch$, since $p$ is always clear from the context.
Notice that $\CH(X)=\CH^*(X)=\CH_*(X)$ is graded, but we often omit the $*$ in the notation meaning the whole
Chow group.

For a field extension $E/F$ we call a cycle $\alpha\in\CH(X_E)$ {\it rational}, if it is defined over $F$,
i.e. lies in the image of the restriction homomorphism $\CH(X)\to\CH(X_E)$.
                                                                           
If $X$ is a flag variety for a split group $G$, then $\CH(X)$ is combinatorial and depends only
on the Dynkin type of $G$ and on the (combinatorial) type of $X$. Therefore $\CH(X)$ does not change under
field extensions and we identify $\CH(X)$ and $\CH(X_E)$ for all $E/F$.

We work in the category of Grothendieck's Chow motives over $F$ with coefficients in $\zz$ or in $\zz/p$ for some
prime number $p$. The construction of this category in given \cite{Ma68} (see also the book \cite{EKM}).
Namely, one starts with the category of correspondences. Its objects are smooth projective varieties over $F$
and the morphisms between two such varieties $X$ and $Y$ are given by
$$\Hom(X,Y)=\bigoplus\CH_{\dim X_i}(X_i\times Y)\otimes R,$$
where $R$ is the coefficient ring and the sum runs over all irreducible components $X_i$ of $X$.
The category of the Chow motives arises as the pseudo-abelian completion of the category of correspondences.
In particular, its objects are pairs $(X,\pi)$ where $X$ is a smooth projective variety over $F$ and $\pi\in\Hom(X,X)$
is a projector, i.e., it satisfies the equation $\pi\circ\pi=\pi$.

Sometimes, we work in a larger category $DM^{eff}_-(F)$ of effective motives of Voevodsky over $F$.
The construction of this category can be found e.g. in \cite{FSV00}.

By $\Mot(X)$ we denote the motive of $X$. For a motive $M$ we denote by $$M\{m\}=M(m)[2m]$$ its {\it Tate twist} by $m$.
We say that the {\it Rost nilpotence} principle holds for $X$, if the kernel of the restriction homomorphism
$$\End(\Mot(X))\to\End(\Mot(X_E))$$ consists of nilpotent elements for all field extensions $E/F$. Notice that the Rost
nilpotence principle holds for all twisted flag varieties, see \cite{CGM05}.

Starting from Section~\ref{sec7} we deal with Chow motives over a smooth base scheme $Z$. The construction
of this category of motives is exactly the same as for Chow motives over a field, but one
starts with smooth projective varieties over $Z$ and takes the products over $Z$ and not over $F$.
\end{ntt}

\begin{ntt}[Equivariant Chow groups]
In the present article we consider the equivariant Chow rings $\CH_G$ of algebraic varieties, see \cite{EG98}.
Moreover, we use the Chow rings of the \'etale classifying spaces of algebraic groups, see \cite{To99}. Notice that
there exists a more general context of Morel--Voevodsky (see \cite{MV99}).

The basic property of the equivariant Chow rings which we use in the present article is the following one.
If $E$ is a $G$-torsor and $H$ is a subgroup of $G$, then one has a sequence
$$H\to E\to E/H\to BH,$$
where $BH$ is the classifying space of $H$, which induces a commutative diagram of pullbacks
\begin{equation}\label{dia11}
 \xymatrix{
 \CH_G(G/H)=\CH_H(\pt)=\CH(BH) \ar@{->}[r] \ar@{->}[rd]  & \CH(E/H) \ar@{->}[d]^{\res}    \\
 & \CH(G/H)
 }
\end{equation}

A wide literature is devoted to the computations of Chow rings of classifying spaces
for different groups. For example, there are results due to Field, Guillot, Molina, Pandharipande, Totaro, Vezzosi, Vistoli, Yagita,
and others, see e.g. \cite{Gui07} and references there.

There exists the category of {\it $G$-equivariant} Chow motive, which can be defined exactly in the same way, as the
category of {\it ordinary} Chow motives with $\CH$ replaced by $\CH_G$.

We say that a motive $M$ over $F$ in the category of (equivariant or ordinary) Chow motives is {\it geometrically split},
if there exists a field extension $E/F$ such that $M_E$ is isomorphic to a (finite) direct sum of Tate motives.
The field $E$ is called a {\it splitting field} of $M$.

For a geometrically split motive $M$ we can define a {\it Poincar\'e series} as
$$P(M,t):=P(M_E,t):=P(A(M_E),t)=\sum_{i\ge 0}\dim A^i(M_E)\cdot t^i\in\zz[[t]],$$
where $A$ is either $\CH$ or $\CH_G$.
Notice that for the ordinary Chow groups $P(M,t)$ is a polynomial, but for the equivariant Chow groups it can
be an infinite power series (see Theorem~\ref{tlift}). The Poincar\'e series does not depend on the choice of a splitting field $E$.
\end{ntt}

\begin{ntt}[Ordinary Rost motives]
One of the most amazing classes of Chow motives is the class of Rost motives.
We define them now. Let $F$ be a field, $n$ be a natural number and $p$ be a prime number with $\mathrm{char}\,F\ne p$.
Consider a non-zero element $$u\in K_n^M(F)/p\simeq H^n(F,\mu_p^{\otimes n}).$$

An indecomposable motive $R=R_{n,p}$ is called the Rost motive for $u$, if for all field extensions $E/F$
$u_E=0$ if and only if $R_E\otimes\zz/p$ is decomposable, in which case we require that it is isomorphic to a direct sum of Tate motives over $E$.

If $u$ is a pure symbol, then the Rost motives are constructed by Rost and Voevodsky. For small values of $n$
and $p$ they are closely related to Pfister quadrics, Severi--Brauer varieties,
Albert algebras and exceptional algebraic groups (see Section~\ref{secaff}).
The Poincar\'e polynomial of a Rost motive corresponding to a pure symbol $u$ equals
$$P(R_{n,p},t)=\frac{1-t^{p\frac{p^{n-1}-1}{p-1}}}{1-t^{\frac{p^{n-1}-1}{p-1}}}=\frac{1-t^{pd}}{1-t^d},$$
if we denote $d:=\frac{p^{n-1}-1}{p-1}$.
\end{ntt}

\section{Generic torsors}

Let $G$ be a linear algebraic group over a field $F$. We can embed $G$ into $\GL_n$ for some $n$ as
a closed subgroup. Then the canonical morphism $\GL_n\to\GL_n/G$ is a $G$-torsor, which we call
a {\it standard classifying torsor}. The generic fiber of this torsor is a $G$-torsor over $F(\GL_n/G)$,
which we call a {\it standard generic torsor} (see \cite[Ch.~1, \S5]{GMS03}).

\begin{lem}\label{lgen}
Let $G$ be a linear algebraic group over a field $F$ and let $E$ be a standard generic $G$-torsor over $F$.
Then $\CH^*(E)=\zz$.
\end{lem}
\begin{proof}
We have the fiber square
\[
 \xymatrix{
 E \ar@{->}[r]^{g} \ar@{->}[d]  & \GL_n \ar@{->}[d]    \\
 \Spec{F(\GL_n/G)} \ar@{->}[r] & \GL_n/G
 }
 \]

By the localization sequence the pullback to the generic fiber $g^*\colon\CH^*(\GL_n)\to\CH^*(E)$
is surjective, since $E$ is open in $\GL_n$.
But $\GL_n$ is an open subvariety of $\mathbb{A}^{n^2}$, and $\CH^*(\mathbb{A}^{n^2})=\zz$. Thus, $\CH^*(E)=\CH^*(\GL_n)=\zz$.
\end{proof}

\section{Rost motives and affine varieties}\label{secaff}

In the present section we abbreviate the Rost motives corresponding to elements in $H^n(F,\mu_p^{\otimes n})$ by $R_{n,p}$.

Assume that $\Char F\ne 2$. Let $a_1,\ldots,a_n\in F^\times$ and let
$q=\langle\!\langle a_1,\ldots,a_n\rangle\!\rangle$ be the $n$-fold Pfister form for the symbol
$(a_1)\cup\ldots\cup(a_n)\in H^n(F,\zz/2)$. Let $q'$ be a codimension one subform of $q$ (a maximal Pfister neigbour).
Consider the respective projective quadrics $Q$ and $Q'$ given by $q=0$ and resp. by $q'=0$ and the affine Pfister quadric
$Q\setminus Q'$.

\begin{thm}[$R_{n,2}$]\label{rn2}
In the above notation the motive of the affine Pfister quadric in the category $\mathrm{DM}^{eff}_-(F)$ is isomorphic
to the Rost motive of the symbol $(a_1)\cup\ldots\cup(a_n)\in H^n(F,\zz/2)$.
\end{thm}
\begin{proof}
Denote by $f\colon Q'\to Q$ the respective embedding of quadrics.
By \cite[Ch.~5, Prop.~3.5.4]{FSV00} there is an exact triangle in $\mathrm{DM}^{eff}_-(F)$ of the form
$$\Mot(Q')\xrightarrow{f_*}\Mot(Q)\to\Mot(Q\setminus Q')\to\Mot(Q')[1].$$

The varieties $Q$ and $Q'$ are projective homogeneous generically split varieties. Their motivic decompositions look as
follows (see \cite{Ro98}, \cite[Section~7]{PSZ08}):

$$\Mot(Q)=\bigoplus_{i=0}^{2^{n-1}} R_{n,2}\{i\},$$
$$\Mot(Q')=\bigoplus_{i=0}^{2^{n-1}-1} R_{n,2}\{i\}.$$

We claim that the map $f_*\colon\Mot(Q')\to\Mot(Q)$ has a left inverse. Indeed, consider the composition of morphisms
$\Mot(Q')\xrightarrow{f_*}\Mot(Q)\xrightarrow{s}\Mot(Q)\xrightarrow{f^*}\Mot(Q')$, where $s$ is the shifting morphism,
which maps $R_{n,2}\{i\}$ to $R_{n,2}\{i-1\}$ for all $i=1,\ldots,2^{n-1}$. Notice that $s\in\CH_{\dim Q+1}(Q\times Q)$,
$f^*\in\CH_{\dim Q-1}(Q\times Q')$ and $f^*\circ s\in\CH_{\dim Q}(Q\times Q')$.

Over a splitting field of $Q$ and $Q'$ this composite morphism is an isomorphism. Since the Rost nilpotence
principle holds for $Q$ and $Q'$ (see \cite[Prop.~9]{Ro98}, \cite[Section~8]{CGM05}), this implies that $f_*$
has a left inverse over the base field $F$ (cf. \cite[Thm.~1.2]{ViZ08}).

Thus, we get another exact triangle, namely,
$$\Mot(Q')\xrightarrow{f_*}\Mot(Q)\to R_{n,2}\xrightarrow{0}\Mot(Q')[1].$$

Since the cone in an exact triangle is unique up to isomorphism, it follows that $\Mot(Q\setminus Q')\simeq R_{n,2}$.
\end{proof}
                                                                                  
Assume now that $\Char F\ne 2,3$.
An {\it Albert algebra} over $F$ is a $27$-dimensional simple exceptional Jordan algebra over $F$ (see e.g. \cite[Ch.~IX]{Inv}).
The variety $Y$ of $1$-dimensional subspaces $Fx$ of $J$ such that $$x^2=0$$ is a projective $15$-dimensional homogeneous
variety of type $\F_4$ under the group $\Aut(J)$. This variety is a twisted form of $\F_4/P_4$ of maximal parabolic subgroups
of type $4$ (the enumeration of simple roots follows Bourbaki); see \cite[27.2]{Fre59}, \cite[10.13]{Ti74}.

An element $x\in J$ is called {\it totally singular}, if $$x\times x=0$$ for the Freudenthal $\times$-product on $J$.
An element $x$ is totally singular, if and only if $x^2=0$ or $x$ is a non-zero primitive idempotent (\cite[p.~364]{Jac68}).

Consider now the variety $X$ of $1$-dimensional totally singular subspaces of $J$. This is a projective
$16$-dimensional variety of inner type $\E_6$ under the group of isometries of the cubic norm of $J$. This variety
is also a twisted form of the variety $\E_6/P_1$ of maximal parabolic subgroups of type $1$ (known also as
the Cayley plane); see \cite[3.2]{SV68}, \cite[3.2]{Ti57}.

Consider the complement $X\setminus Y$. This is a $16$-dimensional affine variety over $F$ of non-zero primitive idempotents of $J$.
This variety is a twisted form of $\F_4/\Spin_9$.

\begin{thm}[$R_{3,3}$]\label{tf4}
In the above notation assume that the Albert algebra $J$ comes from the first Tits construction. Then the motive of this twisted form of $\F_4/\Spin_9$ in the category $\mathrm{DM}^{eff}_-(F)$ is isomorphic
to the Rost motive corresponding to the Serre--Rost invariant $g_3\in H^3(F,\zz/3)$ of the Albert algebra $J$.
\end{thm}
\begin{proof}
The proof is the same as of Theorem~\ref{rn2} with the following motivic decompositions of the varieties $X$ and $Y$
(see \cite{NSZ09}, \cite[Section~7]{PSZ08}):
$$\Mot(X)=\bigoplus_{i=0}^{8} R_{3,3}\{i\},$$
$$\Mot(Y)=\bigoplus_{i=0}^{7} R_{3,3}\{i\}.$$
\end{proof}

Let $p$ be a prime number and let $F$ be a field with $\Char F\ne p$ and containing a primitive $p$-th root of unity. Let $a,b\in F^\times$ and consider the cyclic
algebra $A:=\{a,b\}$ of degree $p$ corresponding to the symbol $(a)\cup(b)\in H^2(F,\mu_p^{\otimes 2})$. Let $X$ be the
product of Severi--Brauer varieties
$\SB(\{a,b\})\times\SB(\{a,b\}^\op)$ of dimension $2(p-1)$. Consider its codimension $1$ subvariety $Y$ of
pairs of right ideals $(I,J)$ of $A$ such that $\dim I=p$, $\dim J=p(p-1)$, and $I\subset J$
(Notice that a right ideal $J$ of $A$
of dimension $p(p-1)$ corresponds to a right ideal $J^o$ of $A^\op$ of dimension $p$ by \cite[1.B]{Inv}).
The variety $Y$ is called the incidence variety. 
The complement $X\setminus Y$ is an affine variety, which is a twisted form of $\SL_p/\GL_{p-1}$.

\begin{thm}[$R_{2,p}$]
In the above notation the motive of $X\setminus Y$ in the category $\mathrm{DM}^{eff}_-(F)$ is isomorphic
to the Rost motive of the symbol $(a)\cup(b)\in H^2(F,\mu_p^{\otimes 2})$.
\end{thm}
\begin{proof}
The proof is the same as of Theorem~\ref{rn2} with the following motivic decompositions of the varieties $X$ and $Y$
(see \cite[Section~7]{PSZ08}):
$$\Mot(X)=\bigoplus_{i=0}^{p-1} R_{2,p}\{i\},$$
$$\Mot(Y)=\bigoplus_{i=0}^{p-2} R_{2,p}\{i\}.$$
\end{proof}

\begin{rem}
The motive $R_{2,p}$ is also isomorphic to the motive of $\SB(\{a,b\})$, hence, can be also realized as the motive
of a projective variety.
\end{rem}

Finally, let $F$ be a field with $\Char F\ne 2$. Consider an octonion algebra $C$ over $F$ (see e.g. \cite[Ch.~VIII]{Inv}, \cite{SV00}).
The norm form of $C$ is a  $3$-fold Pfister form $q$ over $F$. Denote by $Q$ and $Q'$ the projective quadrics given by
$q=0$ and resp. by $q'=0$, where $q'$ is the restriction of $q$ to the $7$-dimensional subspace of $C$ of trace zero
elements.

The group $\Aut(C)$ is a group of type $\G_2$ and is a subgroup of $\Spin(q')$.
It acts on the variety of parabolic subgroups of type $3$ of the group $\Spin(q')$ with $2$ orbits (see \cite{Cu03}). The closed orbit
is isomorphic to a $5$-dimensional Pfister neighbour quadric $Q'$
and is a twisted form of $\G_2/P_1$, and the open
orbit is isomorphic to a $6$-dimensional affine Pfister quadric, which is a twisted form of $\G_2/\SL_3$. Notice
that via an exceptional isomorphism the variety of parabolic subgroups of $\Spin(q')$ of type $3$ is isomorphic to the
$6$-dimensional projective Pfister quadric $Q$.

\begin{thm}[$R_{3,2}$]\label{tg2}
In the above notation the motive of the respective twisted form of $\G_2/\SL_3$ in the category $\mathrm{DM}^{eff}_-(F)$ is isomorphic
to the Rost motive of the symbol in $H^3(F,\zz/2)$ given by the norm form of the octonion algebra $C$.
\end{thm}
\begin{proof}
The proof is the same as of Theorem~\ref{rn2}. Notice that $Q$ is a $6$-dimensional Pfister quadric and $Q'$ is its codimension
$1$ subquadric.
\end{proof}

\section{$G$-equivariant motives}\label{seclift}
Let $G$ be a split semisimple linear algebraic group over a field $F$.
A subgroup $P$ of $G$ is called {\it special}, if $H^1(E,P)=1$ for all field extensions $E/F$.
By \cite[Cor.~6.8]{KM06} if $P$ is a special parabolic subgroup, then every inner twisted form of $G/P$ is generically split.

Let $T$ denote a split maximal torus of $G$, $\widehat T$ the character group of $T$,
and $P$ a special parabolic subgroup of $G$. By $W_P$ denote the Weyl group of $P$.

Consider the equivariant Chow group $\CH_G(G/P\times G/P)$. By the Chernousov--Merkurjev method \cite[Section~6]{GPS12}
there is a decomposition of $\CH_G(G/P\times G/P)$ into a direct sum of groups $\CH_G(G/P')$ for some special
parabolic subgroups $P'$ of $G$. (In \cite{GPS12} we consider the ordinary Chow groups, but the same proof works for
equivariant Chow groups as well.)

Since $P'$ is special, one has  by \cite[Section~3.2]{EG98} $\CH_G(G/P')=\CH(BP')=S(\widehat T)^{W_{P'}}$, the group of $W_{P'}$-invariants
of the symmetric algebra of $\widehat T$. In particular, $\CH_G(G/P')$ and $\CH_G(G/P\times G/P)$ are torsion-free.
Therefore, since by \cite[Prop.~6]{EG98} $$\CH_G(G/P\times G/P)\otimes\qq\subset\CH_T(G/P\times G/P)\otimes\qq,$$ the natural homomorphism
\begin{align}\label{inj}
\CH_G(G/P\times G/P)\to\CH_T(G/P\times G/P)
\end{align}
is injective. Moreover, this homomorphism is compatible
with the composition of correspondences.

\begin{thm}\label{t1}
Let $G$ be a split semisimple linear algebraic group over a field $F$ and $P$ be a special parabolic subgroup of $G$ over $F$. Let $$\Mot(G/P)=\bigoplus_i M_i$$ be a motivic decomposition of the ordinary motive of $G/P$, where $M_i=(G/P,\pi_i)$, $\pi_i\in\CH_{\dim G/P}(G/P\times G/P)$. Assume that all $\pi_i$ lie in the image of the natural homomorphism $$\CH_G(G/P\times G/P)\xrightarrow{\rho}\CH(G/P\times G/P).$$

Then this motivic decomposition of $G/P$ lifts to a motivic decomposition of the equivariant
motive $\Mot_G(G/P)$. Namely, $$\Mot_G(G/P)=\bigoplus_i\widetilde{M_i},$$
where $\widetilde{M_i}=(G/P,\widetilde{\pi_i})$ and $\rho(\widetilde{\pi_i})=\pi_i$ for all $i$.

If $M_i\simeq M_j\{l\}$ for some $i$, $j$, $l$, then $\widetilde{M_i}\simeq\widetilde{M_j}\{l\}$.
\end{thm}
\begin{proof}
Denote $m:=\dim G/P$.
By \cite[Prop.~2.6]{PSZ08} it suffices to check
that the kernel of $$\CH^m_G(G/P\times G/P)\to\Im\Big(\CH^m_G(G/P\times G/P)\to\CH^m(G/P\times G/P)\Big)$$
consists of nilpotent elements.

Since the homomorphism~\eqref{inj} is injective, it suffices to prove that
the kernel of the natural homomorphism
\begin{align}\label{prc}
\CH^m_T(G/P\times G/P)\to\CH^m(G/P\times G/P).
\end{align}
consists of nilpotent elements (with respect to the composition of correspondences).

The description of $\CH_T(G/P\times G/P)$ and of $\CH(G/P\times G/P)$ is given e.g. in \cite[Section~5]{GPS12}. Namely,
$\CH_T(G/P\times G/P)$ (resp. $\CH(G/P\times G/P)$) has a free basis as a module over $\CH_T(\pt)=S(\widehat T)$
(resp. as a module over $\CH(\pt)=\zz$) of the form
$Z_w^T\times Z_u^T$, $w,u\in W/W_P$ (resp. of the form $Z_w\times Z_u$).

The ring $\CH_T(\pt)$ is a subring of the symmetric algebra
$\zz[\w_1,\ldots,\w_n]$, where $\w_i$ are the fundamental weights for $G$ and the projection
$\CH_T(G/P\times G/P)\to\CH(G/P\times G/P)$ sends $f\cdot Z_w^T\times Z_u^T$ to $\bar f\cdot Z_w\times Z_u$,
where $\bar f$ is the constant term of the polynomial $f\in\CH_T(\pt)$.

Therefore the kernel of the projection~\eqref{prc} is generated by the elements of the form $f\cdot Z_w^T\times Z_u^T$
for homogeneous polynomials $f$ without constant term. Notice that codimension of
$f\cdot Z_w^T\times Z_u^T$ equals $\deg f+l(w)+l(u)$, where $l$ is the length of the minimal representative
in the respective coset. In particular, composing two elements from the kernel of \eqref{prc} the degree of the
coefficient will strictly grow. On the other hand, since the total degree is $m$ and there are no elements
in $\CH_T(G/P\times G/P)$ of negative degree, we will arrive at some point to $0$.
\end{proof}

\begin{rem}
Theorem~\ref{t1} holds (with the same proof) for motives with any coefficient ring.
\end{rem}

We apply now this theorem to the case of the variety of Borel subgroups of $G$.

Let $p$ be a prime number.
Let $B$ be a Borel subgroup of $G$ containing $T$. The motivic decomposition of the ordinary motive of
a {\it twisted form} of $G/B$ with $\zz/p$-coefficients is given in \cite[Thm.~5.13]{PSZ08}. It depends on several combinatorial parameters
and a geometric invariant, the $J$-invariant.

In the table below we summarize the combinatorial parameters related to the $J$-invariant and list the degrees of generators
of certain ideal $I_p$ taken from the last column of table 2 of \cite{Kac85}. In the next sections we will give a geometric interpretation
of these degrees.

\begin{center}
\begin{longtable}{l|l|l|l|l|l}
\caption{Parameters of the $J$-invariant}\label{jinv}\\
Group & $p$ & $r$ & $d_i$, $i=1\ldots r$ & $k_i$, $i=1\ldots r$ & degrees for generators of $I_p$, $d_{i,p}$\\
\hline
$\SL_n/\mu_m,\,m|n$       & $p|m$ & $1$           &  $1$  & $p^{k_1}\parallel n$ & $1,\,2,\ldots,\xout{p^{k_1}},\ldots,n$  \\
$\PGSp_n,\,2|n$           & $2$       & $1$                       & $1$    & $2^{k_1}\parallel n$         & $1,\,2,\ldots,\xout{2^{k_1}},\ldots,n$ \\
$\mathrm{O}^+_n$, $n>2$       & $2$       & $[\frac{n+1}{4}]$       & $2i-1$    & $[\log_2\frac{n-1}{2i-1}]$   & $2,\,3,\ldots,[\frac n2]$\\
$\Spin_n$, $n>2$              & $2$       & $[\frac{n-3}{4}]$         & $2i+1$  & $[\log_2\frac{n-1}{2i+1}]$  & $1,\,2,\,3,\ldots,[\frac n2],\,2^s$\  ($2^s<n\le 2^{s+1})$\\
$\PGO_{2n}^+$, $n>1$            & $2$       & $[\frac{n+2}{2}]$      & $1,\,i=1$ & $2^{k_1}\parallel n$         & $1,\,1,\,2,\,3,\ldots,\xout{2^{k_1}},\ldots,n$ \\
                              &           &                 & $2i-3,\,i\ge 2$  &$[\log_2\frac{2n-1}{2i-3}]$   & \\
$\Spin_{2n}^{\pm},\,2\vert n$ & $2$       & $\frac{n}{2}$         & $1,\,i=1$  &$2^{k_1}\parallel n$          & $1,\,2,\,3,\ldots,\xout{2^{k_1}},\ldots,n,\,2^s$\\
                              &           &                  & $2i-1,\,i\ge 2$  & $[\log_2\frac{2n-1}{2i-1}]$ & ($2^s<n\le 2^{s+1})$\\
$\G_2$              & $2$       & $1$     & $3$ & $1$                         & $2,\,3$\\
$\F_4$              & $2$       & $1$     & $3$ & $1$                         & $2,\,3,\,8,\,12$\\
$\E_6$              & $2$       & $1$     & $3$ & $1$                         & $2,\,3,\,5,\,8,\,9,\,12$\\
$\F_4$       & $3$       & $1$            & $4$  & $1$                          & $2,\,4,\,6,\,8$\\
$\E_6^{sc}$       & $3$       & $1$       & $4$  & $1$                          & $2,\,4,\,5,\,6,\,8,\,9$\\
$\E_7$       & $3$       & $1$            & $4$  & $1$                          & $2,\,4,\,6,\,8,\,10,\,14,\,18$\\
$\E_6^{ad}$   & $3$       & $2$                & $1,\,4$  & $2,\,1$                     & $1,\,2,\,4,\,5,\,6,\,8$\\
$\E_7^{sc}$                   & $2$       & $3$                  & $3,\,5,\,9$ & $1,\,1,\,1$    & $2,\,3,\,5,\,8,\,9,\,12,\,14$\\
$\E_7^{ad}$                   & $2$       & $4$                & $1,\,3,\,5,\,9$ & $1,\,1,\,1,\,1$           & $1,\,3,\,5,\,8,\,9,\,12,\,14$\\
$\E_8$                        & $2$       & $4$             & $3,\,5,\,9,\,15$  & $3,\,2,\,1,\,1$             & $2,\,3,\,5,\,8,\,9,\,12,\,14,\,15$\\
$\E_8$                        & $3$       & $2$            & $4,\,10$           & $1,\,1$                    & $2,\,4,\,8,\,10,\,14,\,18,\,20,\,24$\\
$\E_8$                        & $5$       & $1$                        & $6$   & $1$                          &$2,\,6,\,8,\,12,\,14,\,18,\,20,\,24$\\
\hline
\end{longtable}
\end{center}

Let $\xi\in Z^1(E,G)$ be a $1$-cocycle over a field extension $E$ of the base field $F$ such that its $J$-invariant takes its maximal possible value,
i.e. $J_p(\xi)=(k_1,\ldots,k_r)$ (e.g. if $\xi$ corresponds to a generic torsor; see \cite[Example~4.7]{PSZ08}).
Then by \cite[Thm.~5.13]{PSZ08} the motivic decomposition of the respective twisted form of $G/P$ looks as follows:
\begin{align}\label{motdec}
\Mot({}_\xi(G/P))\otimes\zz/p=\bigoplus_{i\ge 0} R_{p}(\xi)\{i\}^{\oplus c_i},
\end{align}
where the Poincar\'e polynomial of $R_p(\xi)$ over a splitting field of $\xi$ equals
\begin{align}\label{f5}
P(\Ch(R_p(\xi)),t)=\prod_{j=1}^r\frac{1-t^{d_jp^{k_j}}}{1-t^{d_j}}
\end{align}
and the integers $c_i$ are the coefficients of the quotient
$$\sum_{i\ge 0}c_it^i=P(\Ch(G/B),t))/P(\Ch(R_p(\xi)),t).$$

Finally, the Poincar\'e polynomial of $\Ch(G/B)$ does not depend on $p$ and equals
\begin{align}\label{f6}
P(\Ch(G/B),t)=\prod_{i=1}^l\frac{1-t^{e_i}}{1-t},
\end{align}
where $e_i$ are the degrees of fundamental polynomial invariants of $G$ and are given in the following table (see \cite[Section~2]{PS10}):
\begin{center}
\begin{longtable}{c|c}
\caption{Degrees of fundamental polynomial invariants}\\
Dynkin type & $e_i$\\
\hline
$\mathrm{A}_m$ & $2,3,\ldots,m+1$\\
$\mathrm{B}_m,\mathrm{C}_m$ & $2,4,\ldots, 2m$\\
$\mathrm{D}_m$ & $2,4,\ldots, 2m-2,m$\\
$\mathrm{E}_6$ & $2,5,6,8,9,12$\\
$\mathrm{E}_7$ & $2,6,8,10,12,14,18$\\
$\mathrm{E}_8$ & $2,8,12,14,18,20,24,30$\\
$\mathrm{F}_4$ & $2,6,8,12$\\
$\mathrm{G}_2$ & $2,6$\
\end{longtable}
\end{center}

\begin{lem}\label{maxj}
The $J$-invariant of $\xi$ is maximal, i.e., equals $(k_1,\ldots,k_r)$, if and only if the image
of the restriction map $\psi_{\xi}\colon\Ch({}_\xi(G/B))\to\Ch(G/B)$ coincides with the image of the
characteristic map $\varphi\colon\Ch(BB)\to\Ch(G/B)$.
\end{lem}
\begin{proof}
The inclusion $\Im\varphi\subset\Im\psi_\xi$ holds in general and is proven in \cite[Thm.~6.4]{KM06}, see also \cite{Gr58}.

For the opposite inclusion, it follows from \cite[Thm.~5.5]{PS10} that the dimension of the image of $\psi_\xi$ (over the field $\zz/p$)
depends only on the value of the $J$-invariant of $\xi$.
By \cite[Thm.~6.4]{KM06} there is a cocycle $\xi'$ (a generic cocycle) such that
$\Im\varphi=\Im\psi_{\xi'}$ and by \cite[Example~4.7]{PSZ08} its $J$-invariant is maximal. It follows that
$\dim\Im\psi_\xi=\dim \Im\psi_{\xi'}=\dim\Im\varphi$ and, thus, $\Im\varphi=\Im\psi_\xi$.

By \cite[Thm.~5.5]{PS10} strictly bigger values
of the $J$-invariant (in the $\preceq$-order of \cite[Section~4.5]{PSZ08}) correspond to strictly smaller dimensions of the image of $\psi_\xi$.
This shows the converse implication. 
\end{proof}

\begin{rem}
The image of $\varphi$ is combinatorial and can be described in terms of the root data of
the split group $G$ only.
\end{rem}

\begin{thm}\label{tlift}
Let $G$ be a split semisimple algebraic group of rank $l$ over a field $F$. Let $e_i$ be the degrees
of the fundamental polynomial invariants of $G$, and $p,r,d_i,k_i$ are the (combinatorial) parameters attached
to the $J$-invariant modulo $p$. Then we have the following equivariant motivic decomposition
\begin{align}\label{motdec2}
\Mot_G(G/B)\otimes\zz/p\simeq\bigoplus_{i\ge 0} R_{p,G}(G)\{i\}^{\oplus c_i},
\end{align}
where $R_{p,G}(G)$ is an indecomposable motive
and the integers $c_i$ satisfy
$$\sum_{i\ge 0}c_it^i=\prod_{i=1}^l\frac{1-t^{e_i}}{1-t}\prod_{j=1}^r\frac{1-t^{d_j}}{1-t^{d_jp^{k_j}}}.$$
The Poincar\'e series of the (equivariant) motive $R_{p,G}(G)$ equals
\begin{align}\label{f7}
\frac{1}{(1-t)^l}\cdot\prod_{j=1}^r\frac{1-t^{d_jp^{k_j}}}{1-t^{d_j}}\cdot\prod_{i=1}^l\frac{1-t}{1-t^{e_i}}=
\prod_{i=1}^l\frac{1}{1-t^{d_{i,p}}}
\end{align}

If $\xi$ is a $1$-cocycle with maximal $J$-invariant mod $p$, then over a splitting field $L$ of $\xi$
the equivariant motive $R_{p,G}(G)$ projects to the ordinary motive $(R_p(\xi))_L$.
\end{thm}
\begin{proof}
We will apply Thm.~\ref{t1}. Namely, we will check that the projectors corresponding to motives in decomposition~\eqref{motdec}
(over a splitting field) lie in the image of the homomorphism $\rho$ of Thm.~\ref{t1}

Consider the following commutative diagram, where the horizontal isomorphisms arise from the
Chernousov--Merkurjev method described in \cite[Section~6]{GPS12}:

\[
 \xymatrix{
 \Ch_G(G/B\times G/B) \ar@{<-}[r]^{\quad\simeq}_{f_G} \ar@{->}[d]^{\rho}  & \bigoplus\Ch_G(G/B)\ar@{->}[d]^{\rho'}    \\
 \Ch(G/B\times G/B) \ar@{<-}[r]^{\quad\simeq}_f & \bigoplus\Ch(G/B)
 }
 \]
(we do not write the precise parameters of $\bigoplus$, since they play no role in the proof).
Observe that $\Ch_G(G/B)=\Ch(BB)=S(\widehat T)\otimes\zz/p$.

Let $L$ be a splitting field of $\xi$.
For the motive $R_p(\xi)=({}_\xi(G/B),\pi)$ the respective projector $\pi_L$ is defined over the
base field. Since the homomorphism $f$ preserves rationality of cycles, the respective element in $\bigoplus\Ch(G/B)$
is also defined over the base field.

But the value of the $J$-invariant of $\xi$ is maximal. Therefore by Lemma~\ref{maxj} the subspace of rational cycles
in $\Ch(G/B)$ coincides with the image of $\rho'$. Thus, the projector $\pi_L$ lies in the image of $\rho$
and we can apply theorem~\ref{t1} and construct an equivariant motive $R_{p,G}(G)$.

The motive $R_{p,G}(G)$ is indecomposable, since the motive $R_p(\xi)$ is indecomposable.

Next we compute the equivariant Poincar\'e series.
Passing to the equivariant Chow groups we have $\Ch_G(G/B)=\Ch(BB)=S(\widehat T)\otimes\zz/p$, which is the algebra of polynomials
in $l$ variables. Thus, the Poincar\'e series of $\Ch_G(G/B)$ equals $\dfrac{1}{(1-t)^l}$.

Now it remains to observe that by decomposition~\eqref{motdec2} $P(\Ch_G(R_{p,G}(G)),t)$ equals
$$\frac{P(\Ch_G(G/B),t)}{\sum_{i\ge 0} c_it^i}.$$
\end{proof}

\begin{example}\label{ex1}
We compute the Poincar\'e series for groups of types $\G_2\mod 2$, $\F_4\mod 3$, and $\E_8\mod 5$.
Substituting in formula~\eqref{f7} the respective parameters $p$, $k_i$, $d_i$ etc. we obtain
$$P(R_{2,\G_2}(\G_2),t)=\frac{1}{(1-t^2)(1-t^3)}$$
$$P(R_{3,\F_4}(\F_4),t)=\frac{1}{(1-t^2)(1-t^4)(1-t^6)(1-t^8)}$$
$$P(R_{5,\E_8}(\E_8),t)=\frac{1}{(1-t^2)(1-t^6)(1-t^8)(1-t^{12})(1-t^{14})(1-t^{18})(1-t^{20})(1-t^{24})}$$

Observe that $P(R_{2,\G_2}(\G_2),t)=P(\Ch(B\SL_3),t)$, since by \cite{To99} $$\Ch(B\SL_3)=\zz/2[x_2,x_3]$$ with $\deg x_i=i$,
and $P(R_{3,\F_4}(\F_4),t)=P(\Ch(B\Spin_9))$, since $$\Ch(B\Spin_9)=\zz/3[y_2,y_4,y_6,y_8]$$ with $\deg y_i=i$.
(Notice that the prime $3$ is not a torsion prime for $\Spin_9$ and so we have the same formula for
$\Ch(B\Spin_9)$ as for $\CH\otimes \qq$. But for any split semisimple group $G$ one has $\CH(BG)\otimes\qq=\qq[x_i, i=1,\ldots,l]$,
where $l$ is the rank of $G$ and $\deg x_i=e_i$ the degrees of fundamental polynomial invariants.)

Besides from this, the codimensions $2,3$ for $\G_2$, $2,4,6,8$ for $\F_4$, and $2,6,8,12,14,18,20,24$ for $\E_8$
are precisely the codimensions (excluded codimension $0$), where the Chow group of the respective ordinary
motive $R_p(\xi)$ is non-zero (see \cite[Thm.~8.1]{KM02}, \cite[Thm.~RM.10]{KM13}, \cite[Cor.~10.8]{Ya12},
\cite[Section~4.1]{Vi07}). These numbers are also the degrees of generators given in the last column of Table~\ref{jinv}.

In the next sections we will give an explanation of these two amazing phenomena.
\end{example}

\section{Equivariant motives of types $\G_2$ and $\F_4$}\label{sec6}
We start explaining the first phenomena, namely, that some equivariant Rost motives can be identified with
the classifying spaces of algebraic groups ($\SL_3$ mod $2$ and $\Spin_9$ mod $3$).

Consider the following diagram of localization sequences (see also diagram~\ref{dia11}):
\[
 \xymatrix{
 \CH_{\G_2}(\G_2/P_1) \ar@{->}[r] \ar@{->}[d] &\CH_{\G_2}(\Spin_7/P_3) \ar@{->}[d] \ar@{->}[r] &\CH_{\G_2}(\G_2/\SL_3) \ar@{->}[d] \ar@{->}[r]& 0\ar@{->}[d]\\
 \CH(E/P_1) \ar@{->}[r] & \CH(X)\ar@{->}[r] & \CH(E/\SL_3)\ar@{->}[r] & 0
 }
 \]
where $E$ is a $\G_2$-torsor over $F$ and $X$ is the respective twisted form of $\Spin_7/P_3$ (a $6$-dimensional
Pfister quadric); see Section~\ref{secaff} and Thm.~\ref{tg2}.

By Thm.~\ref{tg2} $E/\SL_3$ is the (ordinary) Rost motive for a $3$-symbol mod $2$. So, by analogy,
$\CH_{\G_2}(\G_2/\SL_3)$ is the Chow group of an equivariant split Rost motive. But
$\CH_{\G_2}(\G_2/\SL_3)=\CH(B\SL_3)=\zz[x_2,x_3]$, where $\deg x_i=i$. This explains the coincidence of
the equivariant Rost motive and $B\SL_3$.

The ring $\CH_{\G_2}(\G_2/P_1)$ equals $\CH(BP_1)=\CH(BL_1)$, where $L_1$ is the Levi part of the parabolic subgroup $P_1$,
so is isomorphic to $\GL_2$. Thus, $$\CH_{\G_2}(\G_2/P_1)=\CH(B\GL_2)=\zz[y_1,y_2],$$ where $\deg y_i=i$
(see \cite[Sec.~15]{To99}).

The ordinary Chow motive of $E/P_1$ is a direct sum of $3$ ordinary Rost motives:
$$\Mot(E/P_1)=R_{3,2}\oplus R_{3,2}\{1\}\oplus R_{3,2}\{2\}.$$
Exactly as in Thm.~\ref{tlift} we can lift this decomposition to a decomposition for $B\GL_2$.
The image of the (injective) pullback $\CH(B\SL_3)\to\CH(B\GL_2)$ of the natural embedding $\GL_2\to\SL_3$
equals $\zz[y_1^2-y_2,y_1y_2]$. One can show that the respective equivariant decomposition looks as follows:
$$\CH(B\GL_2)=\CH(B\SL_3)\oplus y_1\CH(B\SL_3)\oplus y_1^2\CH(B\SL_3),$$
where we identify $\CH(B\SL_3)$ with its image in $\CH(B\GL_2)$. In particular,
$$y_1^3=y_1y_2+y_1(y_1^2-y_2)$$ and we have even an isomorphism of rings
$$\CH(B\GL_2)\simeq\CH(B\SL_3)[t]/(t^3-tx_2-x_3)$$

The same arguments can be applied to $\F_4$. Namely, one can consider the sequence modulo $3$
$$\Ch_{\F_4}(\F_4/P_4)\to\Ch_{\F_4}(\E_6/P_1)\to\Ch_{\F_4}(\F_4/\Spin_9)$$

By Thm.~\ref{tf4} $\F_4/\Spin_9$ is the Rost motive for a $3$-symbol mod $3$. So, $\Ch_{\F_4}(\F_4/\Spin_9)=\Ch(B\Spin_9)$
is its equivariant analog. In the same way, one can interpret other cases described in Section~\ref{secaff}.
E.g., the equivariant analog for the Rost motive of a $2$-symbol mod $p$ is the space $B\GL_{p-1}$.

\section{Torsion in Chow groups and Rost motives over a general base}\label{sec7}
In this section we explain the second phenomena of example~\ref{ex1} related to torsion in Chow groups.
We will give a geometric interpretation of the numbers in the last column of Table~\ref{jinv}.

In \cite[Thm.~1]{Za07} Zainoulline gives a formula for the canonical dimension of a split simple algebraic group $G$
of rank $l$. Namely, he shows that
$$\cdim_p(G)=N+l-(d_{1,p}+d_{2,p}+\ldots+d_{l,p})=\sum_{i=1}^l e_i-\sum_{i=1}^l d_{i,p},$$
where $N$ is the number of positive roots of $G$ and the integers $d_{i,p}$ are the numbers from the
last column of Table~\ref{jinv} (In \cite{Za07} this result is formulated for odd primes $p$, but proven
for all primes $p$). Notice also that $d_{i,p}=e_i$, if $p$ is not a torsion prime of $G$.

On the other hand, there is another formula for the canonical $p$-dimension of $G$ given in \cite[Cor.~4]{Za07} and
in \cite[Prop.~6.1]{PSZ08}.
Namely, $$\cdim_p(G)=\sum_{i=1}^r d_i(p^{k_i}-1),$$
where the numbers $d_i$, $k_i$ are taken from the previous columns of Table~\ref{jinv}. In particular,
this formula does not involve the numbers from the last column of Table~\ref{jinv}.

Let $G$ be a split semisimple algebraic group over a field $F$ with a split maximal torus $T$
and a Borel subgroup $B$ containing $T$. Fix a prime number $p$ and
consider the characteristic map $S(\widehat T)\otimes\zz/p=\Ch(BB)\xrightarrow{\varphi}\Ch(G/B)$. The kernel of this ring homomorphism is denoted by $I_p$.

As before the ring $S(\widehat T)$ is a polynomial ring and the numbers $d_{i,p}$ are degrees of its generators described
in \cite{Kac85}. Following \cite{Kac85} we denote these generators by $P_1,\ldots,P_l$.

A sequence $A^*\xrightarrow{f} B^*\xrightarrow{g} C^*$ of graded rings is called {\it right exact}, if the homomorphism
$g$ is surjective and $\Ker g$ is the ideal generated by $f(A^+)$, where $A^+$ is the ideal of $A$ generated by
the elements of $A^*$ of strictly positive degree. Equivalently, this means that $g$ induces an
isomorphism
$$
B\otimes_A A/A^+\simeq C.
$$

\begin{lem}\label{tors}
Let $G$ be a linear algebraic group over a field $F$,
$S$ be a special split reductive subgroup of $G$ and $E$ be a $G$-torsor over a smooth base scheme $Z$.
Then the sequence $$\CH^*(BS)\to\CH^*(E/S)\to\CH^*(E)$$
is right exact.
\end{lem}
\begin{proof}
First we consider the case when $S={\mathbb G}_m^n$ is a split torus.
Consider the $G$-torsor $E$ over $Z$ as an $S$-torsor over $X=E/S$. In particular,
it is given by a $1$-cocycle from $Z^1(X,{\mathbb G}_m^n)$. We embed $S\to\GL_n$ and
consider the respective vector bundle $E'$ over $X$.

Then $E$ is an open subvariety of $E'$ with the complement isomorphic to a union
of closed subvarieties isomorphic to $X$, and $\CH^*(X)\simeq\CH^*(E')$.
Then the localization sequence for Chow groups implies the claim in this situation.

We return now to the general situation.
Denote by $T$ a split maximal torus of $S$. We claim that
\begin{equation}\label{fes}
\CH^*(E/S)\otimes_{\CH^*(BS)}\CH^*(BT)=\CH^*(E/T)
\end{equation}

There are right exact sequences
$$\CH^*(E/S)\to\CH^*(E/T)\to\CH^*(S/T)$$
and
$$\CH^*(BS)\to\CH^*(BT)\to\CH^*(S/T)$$
Indeed, since $S$ is special, $E/T\to E/S$ is a Zariski trivial fibration with a cellular fiber $S/T$.
Therefore the first sequence is right exact by \cite[Prop.~1]{EG97}. As for the second one,
using \cite[Lemma~4]{EG97} we see that $\CH^*(BS)=\CH^*(BT)^{W_S}$, where $W_S$
is the Weyl group of $S$. The rest follows from the results of \cite{De73}.

Now we see that the homomorphism from the right to the left in formula~\eqref{fes} is surjective.
But the left- and the right-hand sides of formula~\eqref{fes} are free $\CH^*(E/S)$-modules
of the same rank $|W_S|$. Therefore formula~\eqref{fes} follows.

Applying $-\otimes_{\CH^*(BT)}\zz$ to formula~\eqref{fes} and using that by the above
considerations $\CH^*(E/T)\otimes_{\CH^*(BT)}\zz=\CH^*(E)$, we obtain that
$\CH^*(E)=\CH^*(E/S)\otimes_{\CH^*(BS)}\zz$, as required.
\end{proof}

Let $E$ be a $G$-torsor over $F$ and consider the following commutative diagram of graded rings:
\begin{align*}
 \xymatrix{
 &\Ch(BB) \ar@{=}[d] \ar@{->}[r]^{\varphi_E} &\Ch(E/B) \ar@{->}[d]^{\res} \ar@{->}[r]^{f_E}& \Ch(E)\ar@{->}[d]\\
I_p \ar@{->}[r]& \Ch(BB) \ar@{->}[r]^{\varphi} & \Ch(G/B)\ar@{->}[r] & \Ch(G)
 }
\end{align*}
Assume that $\Ch^*(E)=\zz/p$. Then Lemma~\ref{maxj}, Lemma~\ref{tors} and the diagram above imply
that $\varphi_E$ is surjective, and the $J$-invariant of $E$ is maximal.

Consider a more general situation when $E$ is a $G$-torsor
over a smooth base $Z$ and $\Ch^*(E)=\zz/p$. The map from $E/B$ to $E/B\times_Z E\simeq G/B\times E$ induces
a restriction map
$$
\res\colon\Ch^*(E/B)\to\Ch^*(G/B\times E)=\Ch^*(G/B)\otimes_{\zz/p}\Ch^*(E)
=\Ch^*(G/B),
$$
the isomorphism in the middle is by \cite[Lemma~6.1]{To99}.
We have a diagram
\begin{align}\label{di}
 \xymatrix{
 &\Ch(BB) \ar@{=}[d] \ar@{->>}[r]^{\varphi_E} &\Ch(E/B) \ar@{->}[d]^{\res}\\
I_p \ar@{->}[r]& \Ch(BB) \ar@{->}[r]^{\varphi} & \Ch(G/B)
 }
\end{align}

\begin{lem}\label{kern}
Let $G$ be a split semisimple algebraic group over a field $F$ and $E$ be a $G$-torsor over a smooth base $Z$.
Assume $\Ch(E)=\zz/p$. Then in the above notation $\Ker\res=\varphi_E(I_p)$.
\end{lem}
\begin{proof}
Clear from the diagram above.
\end{proof}

The following lemma follows from \cite[Thm.~4.6]{Ya11}, Lemma~\ref{lgen} and \cite{PSZ08}.
\begin{lem}
Let $G$ be a split semisimple group and let $E$ be a $G$-torsor over a field.
Then $\Ch(E)=\zz/p$ in the following cases:
\begin{itemize}
\item $E$ is a standard generic torsor,
\item or $G$ is a simple group of type $\Phi$, $(\Phi,p)$ equals $(\G_2,2)$, $(\F_4,2)$, $(\F_4,3)$ or $(\E_7,3)$,
and the torsor $E$ is non-trivial.
\end{itemize}
In all these cases the $J$-invariant of $E$ mod $p$ takes its maximal possible value.
\end{lem}

Assume that the $J$-invariant of $E$ mod $p$ takes its maximal possible value.
By \cite{PSZ08} and section~\ref{seclift} the following motivic decompositions hold:
$$\Ch(E/B)=\bigoplus\Ch(R_p(E))$$
$$\Ch(BB)=\Ch_G(G/B)=\bigoplus\Ch_G(R_{p,G}(G))$$
(we do not write the precise parameters for $\bigoplus$) and the equivariant motive $R_{p,G}(G)$
specializes to the ordinary motive $R_p(E)$.

Next we introduce a partial order on $\zz[[t]]$.
For two power series $A=\sum_{i\ge 0} a_it^i, B=\sum_{i\ge 0}b_it^i\in\zz[[t]]$ we write $A\preceq B$,
if $a_i\le b_i$ for all $i\ge 0$. For a power series $C=\sum_{i\ge 0}c_it^i\in\zz[[t]]$, we write $\trunc(C,n)$
for the truncation $\sum_{i=0}^nc_it^i$ up to degree $n$.

\begin{thm}\label{arbbase}
Let $G$ be a split semisimple linear algebraic group over a field $F$ and $P$  be a parabolic subgroup of $G$
over $F$. Let $$\Mot_G(G/P)=\bigoplus_i M_i$$ be a motivic decomposition of the equivariant motive of $G/P$,
where $M_i=(G/P,\pi_i)$, $\pi_i\in\CH_G(G/P\times G/P)$. 
Let $E$ be a $G$-torsor over a smooth scheme $Z$.

Then the above motivic decomposition of $G/P$ induces a motivic decomposition of the ordinary
motive $\Mot(E/P)$ over $Z$. Namely, $$\Mot(E/P)=\bigoplus_i\widetilde{M_i}.$$
If $M_i\simeq M_j\{l\}$ for some $i$, $j$, $l$, then $\widetilde{M_i}\simeq\widetilde{M_j}\{l\}$.
\end{thm}
\begin{proof}
Note that $\CH^*_G(G/P\times G/P)$ acts on $\CH^*_G(G/P\times E\times_Z Z')\simeq
\CH^*(E/P\times_Z Z')$ for any smooth scheme $Z'$ through the natural action on the first factor. In particular, taking $Z'=E/P$ we obtain a map
$$
\CH^*_G(G/P\times G/P)\to\CH^*(E/P\times_Z E/P)
$$
that sends an element $\alpha$ to $\alpha_*(\Delta_{E/P})$, where $*$ denotes the action described above
and $\Delta$ is the diagonal. One can check that this map is a ring homomorphism (with respect to the composition
product), and the claim follows.
\end{proof}

\begin{rem}
The same result (with the same proof) holds for motives with an arbitrary coefficient ring.
\end{rem}

Let $G$ be a split semisimple algebraic group of rank $l$ over a field $F$ and let $E$ be a $G$-torsor over a smooth
scheme $Z$.
Applying now Theorem~\ref{tlift} and Theorem~\ref{arbbase} we get a periodic decomposition $$\Mot(E/B)=\bigoplus_{i\ge 0} R_p(E)\{i\}^{\oplus c_i},$$
where $$\sum_{i\ge 0}c_it^i=\frac{1}{(1-t)^l}\prod_{i=1}^l(1-t^{d_{i,p}}).$$

\begin{thm}\label{ttor}
Let $G$ be a split semisimple algebraic group of rank $l$ over a field $F$ and let $E$ be a $G$-torsor over a smooth
scheme $Z$. Assume that $\Ch(E)=\zz/p$.
Then
$$P\big((\Ch(R_p(E))\cap\Ker\res),t\big)+1\preceq \prod_{i=1}^l\frac{1}{1-t^{d_{i,p}}}.$$

If $Z=\Spec F$, then additionally
$$P\big((\Ch(R_p(E))\cap\Ker\res),t\big)+1\preceq \trunc\big(\prod_{i=1}^l\frac{1}{1-t^{d_{i,p}}},\cdim_p(G)\big).$$

In particular, if $Z=\Spec F$, the group $\Ch(E/B)$ is a finitely generated abelian group.
\end{thm}
\begin{proof}
We have a periodic decomposition $$\Mot(E/B)=\bigoplus_{i\ge 0} R_p(E)\{i\}^{\oplus c_i},$$
where $$\sum_{i\ge 0}c_it^i=\frac{1}{(1-t)^l}\prod_{i=1}^l(1-t^{d_{i,p}}).$$ Therefore taking realizations
we have $$\Ker\res=\bigoplus\big(\Ch(R_p(E))\cap\Ker\res\big)$$ and passing to the Poincar\'e polynomials
$$P(\Ker\res,t)=P\big((\Ch(R_p(E))\cap\Ker\res),t\big)\cdot\big(\sum_{i\ge 0}c_it^i\big).$$

Thus, $P\big((\Ch(R_p(E))\cap\Ker\res),t\big)=\dfrac{P(\Ker\res,t)}{\frac{1}{(1-t)^l}\prod_{i=1}^l(1-t^{d_{i,p}})}$.
By Lemma~\ref{kern} $$P(\Ker\res,t)\preceq P(I_p,t)=P(\Ch(BB),t)-P(\Ch(BB)/I_p).$$
The ideal $I_p$ is generated by the homogeneous polynomials $P_1,\ldots,P_l$ of degrees $d_{1,p},\ldots,d_{l,p}$. Therefore
$$P(\Ch(BB)/I_p)\succeq P(\Ch(BB),t)\prod_{i=1}^l(1-t^{d_{i,p}})=\frac{1}{(1-t)^l}\prod_{i=1}^l(1-t^{d_{i,p}}).$$
Therefore $$P\big((\Ch(R_p(E))\cap\Ker\res),t\big)\preceq\frac{\frac{1}{(1-t)^l}-\frac{1}{(1-t)^l}\prod_{i=1}^l(1-t^{d_{i,p}})}
{\frac{1}{(1-t)^l}\prod_{i=1}^l(1-t^{d_{i,p}})}$$
and the result follows.
\end{proof}

There is a similar method to find a lower bound of the Poincar\'e polynomials of
ordinary motives, cf. Section~\ref{rostbase}.
In cases when $(G,p)$ equal $(\G_2,2)$ and $(\F_4,3)$ this gives an explanation of the mysterious coincidences of Example~\ref{ex1}.
To complete the picture we notice that when $Z$ is a field, we have $\res(\Ch(R_p(E)))=\zz/p$ and $\res(\Ch(E/B))=(\zz/p)^M$ (without grading),
where $M=\prod_{i=1}^l d_{i,p}$. Notice also that $\prod_{i=1}^l e_i=|W|$, the order of the Weyl group of $G$.

\begin{rem}
There exist examples of twisted flag varieties over fields (certain $5$-dimensional quadrics) with {\it infinitely} generated Chow groups, see \cite{KM91}.
\end{rem}

\section{Lower bound on torsion of Rost motives over a general base}\label{rostbase}

Consider the sequence of embeddings $\SL_3\to\G_2\to\GL_7$, where $\SL_3$ is embedded in $\G_2$ as a subsystem
subgroup and $\G_2$ is embedded in $\GL_7$ via its standard representation.

Let $X$ be a smooth scheme over the base field $F$ and $E$ be a $\G_2$-torsor over $X$.
By Lemma~\ref{tors}
the sequence $$\CH^*(B\SL_3)\to\CH^*(E/\SL_3)\to\CH^*(E)$$ is a right exact sequence of graded rings.

By \cite[Prop.~14.2]{To99} the sequence
$$\CH^*(B\GL_7)\xrightarrow{f}\CH^*(B\SL_3)\to\CH^*(\GL_7/\SL_3)$$
is also a right exact sequence of graded rings.

Moreover, by \cite[Section~16]{To99} the map $\CH^*(B\GL_7)\to\CH^*(B\G_2)$ induced by the standard representation
of $\G_2$ is surjective.

The image of $\CH^*(B\G_2)\to\CH^*(B\SL_3)$ equals $\zz[2x_2,x_2^2,x_3^2]$, where as above we identify $\CH^*(B\SL_3)$
and $\zz[x_2,x_3]$, where $x_i$'s have degree $i$. Indeed, by \cite[Remark~7.1]{Gui07} the image of $\CH^*(B\G_2)\to\CH^*(B\SL_3)$ equals the
image of the composite map $$\CH^*(B\Spin_7)\to\CH^*(B\G_2)\to\CH^*(B\SL_3)$$ and the latter image
was computed in \cite[Proof of Prop.~9.1]{Gui07}.

Take now $E=\GL_7$ and $X=\GL_7/\G_2$, so $E\to X$ is a $\G_2$-torsor and $\CH^*(E)=\zz$.
The element $x_2x_3$ does not lie in the ideal generated
by the elements of positive degree in the image of $f$.
Therefore this element maps to a non-zero element in $\CH^5(\GL_7/\SL_3)$.
This element is a torsion element, since $2x_2$ belongs to the image of $f$.

Notice that by the above considerations $B\SL_3$ is an equivariant Rost motive (for a $3$-symbol modulo $2$),
and $E/\SL_3=\GL_7/\SL_3$ is an analog of the ordinary Rost motive over $X=\GL_7/\G_2$.

The same analysis shows that $$\CH^2(\GL_7/\SL_3)=\CH^5(\GL_7/\SL_3)=\zz/2\zz,$$
$$\CH^0(\GL_7/\SL_3)=\CH^3(\GL_7/\SL_3)=\zz,$$
and all other Chow groups of $\GL_7/\SL_3$ are equal to $0$.

Thus, we have shown the following proposition about a lower bound on torsion (see Theorem~\ref{ttor} for an upper bound).
\begin{prop}
There exists a $\G_2$-torsor $E$ over certain base scheme $X$ such that
for $R=E/\SL_3$ one has $\CH^i(R)=\begin{cases}
\zz, \text{ if }i=0,3;\\
\zz/2, \text{ if }i=2,5;\\
0, \text{ otherwise.}
\end{cases}$
\end{prop}

Recall that for an ordinary Rost motive $R'$ of type $(3,2)$ over a {\it field} one has the same formula with the only difference
that $\CH^5(R')=0$, and by Theorem~\ref{tg2} $R'\simeq E/\SL_3$ for a $\G_2$-torsor $E$ over a field.
So, by analogy the motive $R$ from the statement of the above proposition can be viewed as a Rost
motive over the base $X=\GL_7/\G_2$, and it illustrates a new phenomenon which is not visible for
Rost motives over fields.


\begin{thebibliography}{NPSZ15}

\bibitem[CGM05]{CGM05}
{\sc V.~Chernousov, S.~Gille, A.~Merkurjev},
Motivic decomposition of isotropic projective homogeneous varieties,
{\em Duke Math. J.} {\bf 126} (2005), 137--159.

\bibitem[Cu03]{Cu03} {\sc S.~Cupit-Foutou}, Classification of two-orbit varieties,
{\it Comm. Math. Helv.} {\bf 78} (2003), no.~2, 245--265.

\bibitem[De73]{De73} {\sc M.~Demazure}, Invariants sym\'etriques entiers des groupes de Weyl
et torsion, {\it Invent. Math} {\bf 21} (1973), 287--301.

\bibitem[EKM]{EKM} {\sc R.~Elman, N.~Karpenko, A.~Merkurjev}, The algebraic and geometric theory of quadratic forms,
{\em Colloquium Publications}, vol.~56, American Mathematical Society, Providence, RI, 2008.

\bibitem[EG97]{EG97}
{\sc D.~Edidin, W.~Graham}, 
Characteristic classes in the Chow ring, 
{\it J. Algebraic Geom.} {\bf 6} (1997), 431--443.

\bibitem[EG98]{EG98} {\sc D.~Edidin, W.~Graham}, Equivariant intersection theory (with an appendix by
A.~Vistoli: The Chow ring of $\mathcal{M}_2$),
{\it Invent. Math.} {\bf 131} (1998), no.~3, 595--634.

\bibitem[FSV00]{FSV00} {\sc E.~Friedlander, A.~Suslin, V.~Voevodsky}, Cycles, transfers and motivic
homology theories, Princeton Univ. Press, 2000.

\bibitem[Fre59]{Fre59} {\sc H.~Freudenthal}, Beziehungen der $\E_7$ und $\E_8$ zur Oktavenebene. VIII, IX,
{\it Nederl. Akad. Wetensch.} {\bf 62} (1959), 447--474.

\bibitem[Ful]{Ful} {\sc W.~Fulton}, Intersection theory. Second Edition,
Ergebnisse der Mathematik und ihrer Grenz\-gebiete, 3. Folge. A Series of Modern Surveys in Mathematics {\bf 2},
Springer-Verlag, Berlin, 1998.

\bibitem[GMS03]{GMS03}
{\sc S.~Garibaldi, A.~Merkurjev, J-P.~Serre},
Cohomological invariants in Galois cohomology,
AMS, Providence, RI, 2003.

\bibitem[GPS12]{GPS12}
{\sc S.~Garibaldi, V.~Petrov, N.~Semenov},
Shells of twisted flag varieties and the Rost invariant,
Preprint 2012, to appear in Duke Math. J. Available from {\tt http://arxiv.org/abs/1012.2451}

\bibitem[Gr58]{Gr58}
{\sc A.~Grothendieck}, La torsion homologique et les sections rationnelles,
Expos\'e 5 in Anneaux de Chow et applications, S\'eminaire Claude Chevalley, 1958.

\bibitem[Gui07]{Gui07} {\sc P.~Guillot}, The Chow rings of $\G_2$ and $\Spin(7)$, {\it J. reine angew. Math.}
{\bf 604} (2007), 137--158.

\bibitem[Jac68]{Jac68} {\sc N.~Jacobson}, Structure and representations of Jordan algebras, {\it AMS Colloquium Publications},
vol.~39, AMS 1968.

\bibitem[Kac85]{Kac85} {\sc V.~Kac}, Torsion in cohomology of compact Lie groups and Chow
rings of reductive algebraic groups, {\it Invent. Math.} {\bf 80} (1985), 69–-79.

\bibitem[KM91]{KM91} {\sc N.~Karpenko, A.~Merkurjev}, Chow groups of projective quadrics,
{\it Leningrad (St. Petersburg) Math. J.} {\bf 2} (1991), no. 3, 655--671.

\bibitem[KM02]{KM02} {\sc N.~Karpenko, A.~Merkurjev}, Rost projectors and Steenrod operations,
{\it Doc. Math.} {\bf 7} (2002), 481--493.

\bibitem[KM06]{KM06} {\sc N.~Karpenko, A.~Merkurjev},
Canonical $p$-dimension of algebraic groups,
{\it Adv. Math.} {\bf 205} (2006), 410--433. 

\bibitem[KM13]{KM13} {\sc N.~Karpenko, A.~Merkurjev}, On standard norm varieties,
{\it Ann. Sci. Ec. Norm. Sup\'er.} (4) {\bf 46} (2013), 175--214.

\bibitem[KMRT]{Inv}
{\sc M-A.~Knus, A.~Merkurjev, M.~Rost, J-P.~Tignol},
The book of involutions,
{\em Colloquium Publications}, vol.~44, AMS 1998.

\bibitem[Ma68]{Ma68} {\sc Y.~Manin}, Correspondences, motives and monoidal transformations,
{\em Math. USSR Sbornik} {\bf 6} (1968), 439--470.

\bibitem[MV99]{MV99} {\sc F.~Morel, V.~Voevodsky}, $A^1$-homotopy theory of schemes, {\it Publ. Math. IH\'ES}
{\bf 90} (1999), 45--143.

\bibitem[NPSZ15]{NPSZ15} {\sc A.~Neshitov, V.~Petrov, N.~Semenov, K.~Zainoulline},
Motivic decompositions of twisted flag varieties and representations of Hecke-type algebras,
Preprint 2015, {\tt http://arxiv.org/abs/1505.07083}

\bibitem[NSZ09]{NSZ09}
{\sc S.~Nikolenko, N.~Semenov, K.~Zainoulline},
Motivic decomposition of anisotropic varieties of type $\mathrm{F}_4$ into generalized Rost motives,
{\em J. of K-theory} {\bf 3} (2009), no.~1, 85--102.

\bibitem[PS10]{PS10}
{\sc V.~Petrov, N.~Semenov}, Generically split projective homogeneous varieties,
{\em Duke Math. J. } {\bf 152} (2010), 155--173.

\bibitem[PSZ08]{PSZ08}
{\sc V.~Petrov, N.~Semenov, K.~Zainoulline},
$J$-invariant of linear algebraic groups, {\em Ann. Sci. \'Ec. Norm. Sup.} {\bf 41} (2008), 1023--1053.

\bibitem[PY14]{PY14} {\sc A.~Pirutka, N.~Yagita},
Note on the counterexamples for the integral Tate conjecture over finite fields,
Preprint 2014. Available from {\tt http://arxiv.org/abs/1401.1620}

\bibitem[Ro98]{Ro98} {\sc M.~Rost}, The motive of a Pfister form, Preprint 1998, available from {\tt http://www.math.uni-bielefeld.de/\~{}rost}

\bibitem[Ro07]{Ro07} {\sc M.~Rost}, On the basic correspondence of a splitting
variety, Preprint 2007. Available from {\tt http://www.math.uni-bielefeld.de/\~{}rost}

\bibitem[Spr]{Springer} {\sc T.A.~Springer}, Linear algebraic groups. Second Edition, Birkh\"auser Boston, 2009.

\bibitem[SV68]{SV68} {\sc T.A.~Springer, F.D. Veldkamp}, On Hjelmslev-Moufang planes, {\it Math. Z.}  {\bf 107} (1968),
249--263.

\bibitem[SV00]{SV00} {\sc T.A.~Springer, F.D. Veldkamp}, Octonions, Jordan algebras and exceptional groups,
Springer, Berlin, Heidelberg, 2000.

\bibitem[Ti57]{Ti57} {\sc J.~Tits}, Sur la g\'eom\'etrie des $R$-espaces, {\it J. Math. Pures Appl.} {\bf 36} (1957),
17--38.

\bibitem[Ti74]{Ti74} {\sc J.~Tits}, Buildings of spherical type and finite $BN$-pairs, {\it Lecture Notes in Math.},
vol.~386, Springer, 1974.

\bibitem[To99]{To99} {\sc B.~Totaro}, The Chow ring of a classifying space, in {\it Algebraic K-Theory}, ed. W.~Raskind and
C.~Weibel, Proceedings of Symposia in Pure Mathematics, vol.~67, American Mathematical Society, 1999, 249--281.

\bibitem[Vi07]{Vi07} {\sc A.~Vishik}, Symmetric operations in algebraic cobordisms, {\it Adv. Math.} {\bf 213} (2007), 489--552.

\bibitem[ViZ08]{ViZ08} {\sc A.~Visik, K.~Zainoulline}, Motivic splitting lemma, {\it Doc. Math.} {\bf 13} (2008), 81--96.

\bibitem[Ya11]{Ya11} {\sc N.~Yagita}, Note on Chow rings of nontrivial $G$-torsors over a field,
{\it Kodai Math. J.} {\bf 34} (2011), no.~3, 339--535.

\bibitem[Ya12]{Ya12} {\sc N.~Yagita}, Algebraic $BP$-theory and norm varieties, {\it Hokkaido Math. J.}
{\bf 41}, no.~2 (2012), 157--316.

\bibitem[Vo03]{Vo03} {\sc V.~Voevodsky}, {Motivic cohomology with $\zz/2$-coefficients}, {\it Publ. Math.,
Inst. Hautes \'Etud. Sci.} {\bf 98} (2003), 59--104.

\bibitem[Vo11]{Vo11} {\sc V.~Voevodsky}, {On motivic cohomology with $\zz/l$-coefficients},
{\it Ann. Math.} {\bf 174} (2011), no.~1, 401--438.

\bibitem[Za07]{Za07} {\sc K.~Zainoulline}, Canonical $p$-dimensions of algebraic groups and degrees of basic polynomial
invariants, {\it London Math. Bull.} {\bf 39} (2007), 301--304.

\end{thebibliography}
\end{document}